\documentclass{amsart}

\usepackage{amsfonts, amsmath, amssymb, pifont}
\usepackage{mathtools}
\usepackage{hyperref}
\usepackage{enumitem}
\usepackage{todonotes}

\usepackage{xcolor}
\usepackage{tikz}
\usetikzlibrary{calc}

\newlength{\superthick}
\newlength{\cornerradius}
\tikzstyle{corner}=[rounded corners=\cornerradius]
\tikzstyle{string}=[line width=\superthick]
\tikzstyle{dot}=[circle, inner sep=0pt, minimum size=7pt]

\usepackage[style=ams-alphabetic,
	        backend=bibtex,
			url=false,
			doi=false,
			isbn=false,
			giveninits=true,
            maxbibnames=4,
            maxalphanames=4]{biblatex}

\renewbibmacro{in:}{%
	\ifentrytype{article}{}{\printtext{\bibstring{in}\space}}}

\DeclareFieldFormat
[misc,article,inbook,incollection,inproceedings,patent,thesis,unpublished]
{title}{#1\isdot}

\AtEveryBibitem{%
  \ifentrytype{book}{%
    \clearfield{pages}%
  }{%
  }%
}

\renewbibmacro{volume+number+eid}{%
	\iffieldundef{number}{}{\printtext[parens]{%
			\printfield{number}}}
	\printfield[bold]{volume}
	\setunit{\space}%
	\printfield{eid}}

\renewbibmacro*{series+number}{%
	\printfield{series}%
	\setunit*{\addspace}%
	\printfield{number}%
	\iffieldundef{note}{\newunit\nopunct}{\newunit}}

\DeclareFieldFormat[article]{pages}{#1}

\renewbibmacro*{publisher+location+date}{%
	\printlist{location}%
	\iflistundef{publisher}
	{\setunit*{\addcomma\space}}
	{\setunit*{\addcolon\space}}%
	\printtext[parens]{\printlist{publisher}%
		\setunit*{\addcomma\space}%
		\usebibmacro{date}}
	\newunit}

\renewbibmacro*{location+date}{%
	\printlist{location}%
	\setunit*{\space}%
	\printtext[parens]{\usebibmacro{date}}%
	\newunit}

\NewBibliographyString{submitted}
\DefineBibliographyStrings{english}{%
  submitted = {submitted},
}

\theoremstyle{plain}
\newtheorem{thm}{Theorem}[section]
\newtheorem*{thm*}{Theorem}
\newtheorem{clm}[thm]{Claim}
\newtheorem{lem}[thm]{Lemma}
\newtheorem*{lem*}{Lemma}

\newtheorem*{conj*}{Conjecture}
\newtheorem{cor}[thm]{Corollary}
\newtheorem*{cor*}{Corollary}
\newtheorem{prop}[thm]{Proposition}
\newtheorem{thmletter}{Theorem}[section]

\theoremstyle{definition}
\newtheorem{defn}[thm]{Definition}
\newtheorem{exam}[thm]{Example}

\theoremstyle{remark}
\newtheorem{rem}[thm]{Remark}

\numberwithin{equation}{section}

\newcommand{\defnemph}[1]{\emph{#1}}

\newcommand{\ZZ}{\mathbb{Z}}
\newcommand{\QQ}{\mathbb{Q}}

\DeclareMathOperator{\Hom}{Hom}

\DeclareMathOperator{\Sym}{Sym}

\begin{document}

\title[Remarks on realizations]{Some remarks on realizations defining diagrammatic Hecke categories}
\author{Amit Hazi}
\address{School of Mathematics\\
University of Leeds\\
Leeds\\
LS2 9JT\\
United Kingdom}

\email{\href{mailto:A.Hazi@leeds.ac.uk}{A.Hazi@leeds.ac.uk}}

\begin{abstract}
Every diagrammatic Hecke category is constructed from a realization, which generalizes the notion of a reflection representation of a Coxeter group. 
Common assumptions on realizations to ensure that the resulting categories are well behaved include Demazure surjectivity and the parabolic property (for anti-spherical quotients). 
We show that these assumptions are in fact unnecessary. 
We also give a non-inductive description of the rotational scalars for unbalanced realizations, as well as a simpler criterion to check for balancedness. 
\end{abstract}


\maketitle

\section*{Introduction}

Let $(W,S)$ be a Coxeter system, and let $\mathbb{H}:=\mathbb{H}(W,S)$ be the corresponding Iwahori--Hecke algebra. 
A \defnemph{Hecke category} is a graded additive monoidal category $\mathcal{H}$ which categorifies $\mathbb{H}$. 
Such categories (and their categorical actions) are ubiquitous in Lie theory. 
In characteristic $0$ their behavior is often determined by Kazhdan--Lusztig theory, which gives insight into how Kazhdan--Lusztig polynomials control the categories they act on (e.g.~the Kazhdan--Lusztig conjectures for the BGG category $\mathcal{O}$).
In positive characteristic, Hecke categories are used to \emph{define} $p$-Kazhdan--Lusztig theory, providing a replacement for Kazhdan--Lusztig polynomials with applications to modular representation theory. 

As with the Hecke algebra itself, the most general construction of a Hecke category is via a (diagrammatic) presentation in terms of generators and relations due to Elias--Williamson \cite{EW-SoergelCalc}. 
The fundamental input data for such a presentation is a \defnemph{realization} of the Coxeter system $(W,S)$. 
Realizations generalize the notion of a reflection representation (specifically that of a \defnemph{reflection faithful} representation) with labeled roots and coroots. 
Such representations are the foundation for Soergel's original algebraic categorification of the Hecke algebra \cite{soergel-unzerlegbareBimoduln}.
In the diagrammatic setting, realizations were first introduced by Elias in \cite{dihedralcathedral}. 

In many applications the underlying realization is the action of a (finite or affine) Weyl group on a Cartan subalgebra, which is almost always well behaved. 
However, finding the correct completely general definition is surprisingly subtle. 
For example, 7 years after introducing general realizations Elias--Williamson showed that the original definition contained an incorrect inference \cite{ew-loccalc}; this error was subsequently corrected by the author in \cite{2colJW}. 
Moreover, in the literature the relatively weak condition of \defnemph{Demazure surjectivity} is often freely assumed to ensure that the corresponding diagrammatic category actually categorifies the Hecke algebra. 
For \defnemph{anti-spherical} quotients of Hecke categories, Libedinsky--Williamson require an additional strong assumption (the \defnemph{parabolic property}) which often fails in positive characteristic \cite{LW-antispher}. 
Both of these conditions add complexity to an already complicated theory.

The goal of this paper is to simplify the theory of diagrammatic Hecke categories by showing that these assumptions are broadly unnecessary. 
Our main results are as follows.

\begin{thmletter}[{Corollary \ref{cor:cinchedSoergelcatthm}}] \label{thm:intro1}
The diagrammatic category constructed by Elias--Williamson always admits a quotient which categorifies the Hecke algebra.
\end{thmletter}

\begin{thmletter}[{Corollary \ref{cor:antispherSoergelcatthm}}] \label{thm:intro2}
The anti-spherical quotient of a Hecke category (\textit{\`{a} la} Libedinsky--Williamson) always categorifies the anti-spherical module for the Hecke algebra.
\end{thmletter}

In addition, we provide an easier way to check for \defnemph{balancedness}, another assumption on the realization which simplifies some of the notation in the diagrammatic Hecke category.

\begin{thmletter}[{Corollary \ref{cor:balancedcriterion}}]
A realization is balanced if and only if for all distinct $s,t \in S$ for which $st$ has odd order $m$, the $(s,t)$th and $(t,s)$th entries of the Cartan matrix coincide, and satisfy the minimal polynomial of $-2\cos(\pi/m)$ over $\ZZ$. 
\end{thmletter}

This follows from a calculation of the general rotational scalar $[m-1]_s$ in Theorem \ref{thm:mminus1twocol}, which has applications to unbalanced realizations.

\subsection*{Acknowledgments}

The author would like to thank Chris Bowman, Maud De Visscher, and Emily Norton for providing the inspiration for this paper while working on \cite{BHN-modularWeylKac} and \cite{bdhn-hsp}. 
The author would also like to thank Ben Elias for invaluable comments on a previous version of this paper.
The author is additionally grateful for financial support from EPSRC (EP/W007509/1).

\section{Realizations from Cartan matrices}

In this section we will give a simplification of the usual definition of a realization, which will help us prove many of our later results in this paper. 

\begin{defn} \label{defn:rlz}
Let $\Bbbk$ be a commutative integral domain, and let $(W,S)$ be a Coxeter system.
\begin{enumerate}
 \item A \defnemph{$\Bbbk$-valued Cartan matrix} for $(W,S)$ is an $S \times S$ matrix $(a_{st})_{s,t \in S}$ with entries in $\Bbbk$ satisfying the following conditions:
 \begin{enumerate}[label=(\roman*)]
 \item $a_{ss}=2$ for all $s \in S$;
 
 \item For all distinct $s,t \in S$ such that $m_{st}<\infty$,
 \begin{equation} \label{eq:minpolycond}
 \begin{aligned}
 \Psi_{m_{st}}(a_{st} a_{ts})& =0 & &\text{if $m_{st}>2$,} \\
 a_{st}=a_{ts}& =0 & & \text{if $m_{st}=2$,}
 \end{aligned}
 \end{equation}
 where $\Psi_{m_{st}}$ denotes the minimal polynomial of $4\cos^2(\pi/m_{st})$ over $\QQ$ (see \cite[(15)]{2colJW}).
 \end{enumerate}
 
 \item A \defnemph{realization} of $(W,S)$ over $\Bbbk$ consists of a free finite-rank $\Bbbk$-module $V$ along with subsets $\{\alpha_s : s \in S\} \subset V$ and $\{\alpha_s^\vee : s \in S\} \subset V^* = \Hom(V,\Bbbk)$ (called simple roots and coroots respectively) such that the matrix $(\langle \alpha_s^\vee, \alpha_t\rangle)_{s,t \in S}$ is a $\Bbbk$-valued Cartan matrix of $(W,S)$.
 
 \item Let $(V,\{\alpha_s\},\{\alpha_s^\vee\})$ and $(V',\{\alpha'_s\},\{(\alpha'_s)^\vee\})$ be realizations of $(W,S)$ over $\Bbbk$. 
 A \defnemph{morphism of realizations} is a $\Bbbk$-linear map $\phi:V \rightarrow V'$ such that $\phi(\alpha_s)=\alpha'_s$ and $\phi^\ast((\alpha'_s)^\vee)=\alpha_s^\vee$ for all $s \in S$. 
 Let $U \leq V'$ denote the image of $\phi$, and identify its dual $U^\ast$ with the quotient $V'^\ast/U^\circ$ (where $U^\circ \leq V'^\ast$ denotes the annihilator of $U$). 
 The \defnemph{image realization} is given by $(U,\{\alpha'_s\},\{(\alpha'_s)^\vee+U^\circ\})$. 
 Dually, let $K \leq V$ denote the kernel of $\phi$. 
 The \defnemph{coimage realization} is given by $(V/K,\{\alpha_s+K\},\{\alpha_s^\vee\})$, and is isomorphic to the image realization.
\end{enumerate}
\end{defn}

\begin{defn} \label{defn:2colqnum}
Let $A=\ZZ[x_s,x_t]$ be the integral polynomial ring in two variables. 
The \defnemph{(generic) two-colored quantum numbers} are defined as follows. 
First set $[0]_s=[0]_t=0$, $[1]_{s}=[1]_{t}=1$, $[2]_{s}=x_s$, and $[2]_{t}=x_t$. 
For $n>1$ we inductively define
\begin{align}
[n+1]_{s}& =[2]_{s} [n]_{t} - [n-1]_{s} \text{,} & [n+1]_{t} & =[2]_{t} [n]_{s} - [n-1]_{t} \text{.} \label{eq:twocolqnum}
\end{align}

Let $(a_{st})_{s,t \in S}$ be a $\Bbbk$-valued Cartan matrix for a Coxeter system $(W,S)$. 
For distinct $s,t \in S$ we define the corresponding two-colored quantum numbers in $\Bbbk$ by specializing to $x_s=-a_{st}$ and $x_t=-a_{ts}$.
\end{defn}

\begin{rem} \hfill
\begin{enumerate}
\item For any positive integer $m$, the real numbers $2\cos(\pi/m)$ and $4\cos^2(\pi/m)$ are algebraic integers, so $\Psi_m$ in fact has integer coefficients, and thus makes sense over any ring $\Bbbk$. 

\item Equation \eqref{eq:minpolycond} is equivalent to Abe's condition \cite[Assumption~1.1]{abe-homBS} on the vanishing of certain two-colored quantum binomial coefficients \cite[Theorem~3.8]{2colJW}. 
We feel that the minimal polynomial vanishing condition is easier to state and check in practice, especially for crystallographic Coxeter groups. 
%
\end{enumerate}
\end{rem}

In Definition \ref{defn:rlz} we have removed the usual condition that a realization should be a representation of $W$. 
Somewhat surprisingly this turns out to be unnecessary. 
This observation is not original but appears on the whole to have been overlooked.
For completeness we will state and prove it carefully.

\begin{prop}[{cf.~\cite[\S A.2]{dihedralcathedral}}] \label{prop:rlz=cartanmat}
Let $V$ be a realization as defined above. 
Then the assignments
\begin{align*}
s(x)& =x-\langle \alpha_s^\vee, x\rangle \alpha_s \text{,} & s(y)& =y-\langle y,\alpha_s \rangle \alpha_s^\vee
\end{align*}
for any $s \in S$, $x \in V$, and $y \in V^\ast$ define dual representations of $W$ on $V$ and $V^\ast$ respectively.
\end{prop}

\begin{proof}
We will first show that the first assignment defines a representation of $W$ on $V$. 
It is enough to check the defining relations of the Coxeter group on the action. 

Clearly for all $s \in S$ and $x \in V$ we have
\begin{align*}
s(s(x))& =s(x-\langle \alpha_s^\vee, x\rangle \alpha_s) \\
& =(x-\langle \alpha_s^\vee, x\rangle \alpha_s)-\langle \alpha_s^\vee, x\rangle (\alpha_s-\langle \alpha_s^\vee,\alpha_s\rangle \alpha_s) \\
& =(x-\langle \alpha_s^\vee, x\rangle \alpha_s)-\langle \alpha_s^\vee, x\rangle (\alpha_s-2\alpha_s) \\
& =(x-\langle \alpha_s^\vee, x\rangle \alpha_s)+\langle \alpha_s^\vee, x\rangle \alpha_s \\
& =x \text{.}
\end{align*}

Now suppose $s,t \in S$ are distinct with $m_{st}<\infty$. 
Write $k_t$ for the unique word $\dotsm tst$ which is an alternating product of $s$ and $t$ of length $k$ which ends with $t$.
For any $x \in V$ and any integer $k \geq 0$, we claim that
\begin{align}
\begin{split}
(2k)_t(x) & =x - [k]_s [k]_t \langle \alpha_s^\vee, x\rangle \alpha_s - [k]_s [k+1]_s \langle \alpha_t^\vee, x\rangle \alpha_s \label{eq:st-act-even}\\
& \quad \quad {} - [k]_t [k-1]_t \langle \alpha_s^\vee, x\rangle \alpha_t - [k]_s [k]_t \langle \alpha_t^\vee, x\rangle \alpha_t \text{,}
\end{split} \\
\begin{split}
(2k+1)_t(x)& =x - [k]_s [k]_t \langle \alpha_s^\vee, x\rangle \alpha_s - [k]_s [k+1]_s \langle \alpha_t^\vee, x\rangle \alpha_s  \label{eq:st-act-odd} \\
& \quad \quad {} - [k]_t [k+1]_t \langle \alpha_s^\vee, x\rangle \alpha_t - [k+1]_s [k+1]_t \langle \alpha_t^\vee, x\rangle \alpha_t \text{.} 
\end{split}
\end{align}
We will prove this by induction. 

The base cases with $k=0$ are easy to check.
Suppose the result holds for $2k$. 
Then we have
\begin{align*}
(2k+1)_t(x)& =t(2k)_t(x) \\
& =t(x - [k]_s [k]_t \langle \alpha_s^\vee, x\rangle \alpha_s - [k]_s [k+1]_s \langle \alpha_t^\vee, x\rangle \alpha_s \\
& \quad \quad {} - [k]_t [k-1]_t \langle \alpha_s^\vee, x\rangle \alpha_t - [k]_s [k]_t \langle \alpha_t^\vee, x\rangle \alpha_t) \\
& =x - [k]_s [k]_t \langle \alpha_s^\vee, x\rangle \alpha_s - [k]_s [k+1]_s \langle \alpha_t^\vee, x\rangle \alpha_s \\
& \quad \quad {} + ([k]_t [k-1]_t - [2]_t [k]_s [k]_t)\langle \alpha_s^\vee, x\rangle \alpha_t \\
& \quad \quad {} + ([k]_s [k]_t - [2]_t [k]_s [k+1]_s - 1)\langle \alpha_t^\vee, x\rangle \alpha_t \\
& =x - [k]_s [k]_t \langle \alpha_s^\vee, x\rangle \alpha_s - [k]_s [k+1]_s \langle \alpha_t^\vee, x\rangle \alpha_s \\
& \quad \quad {} - [k]_t [k+1]_t \langle \alpha_s^\vee, x\rangle \alpha_t - ([k]_s [k+2]_t + 1)\langle \alpha_t^\vee, x\rangle \alpha_t \\
& =x - [k]_s [k]_t \langle \alpha_s^\vee, x\rangle \alpha_s - [k]_s [k+1]_s \langle \alpha_t^\vee, x\rangle \alpha_s \\
& \quad \quad {} - [k]_t [k+1]_t \langle \alpha_s^\vee, x\rangle \alpha_t - [k+1]_s [k+1]_t \langle \alpha_t^\vee, x\rangle \alpha_t \text{.}
\end{align*}
Here we have used the definition of two-colored quantum numbers, as well as the identity
\begin{equation}
[k]_s [k+2]_t + 1 = [k+1]_s [k+1]_t \text{.}
\end{equation}

Similarly, if the result holds for $2k+1$, then we have
\begin{align*}
(2k+2)_t(x)& =s(2k+1)_t(x) \\
& =s(x - [k]_s [k]_t \langle \alpha_s^\vee, x\rangle \alpha_s - [k]_s [k+1]_s \langle \alpha_t^\vee, x\rangle \alpha_s \\
& \quad \quad {} - [k]_t [k+1]_t \langle \alpha_s^\vee, x\rangle \alpha_t - [k+1]_s [k+1]_t \langle \alpha_t^\vee, x\rangle \alpha_t) \\
& =x + ([k]_s [k]_t - [2]_s [k]_t [k+1]_t - 1)\langle \alpha_s^\vee, x\rangle \alpha_s \\
& \quad \quad {} + ([k]_s [k+1]_s - [2]_s [k+1]_s [k+1]_t)\langle \alpha_t^\vee, x\rangle \alpha_s \\
& \quad \quad {} - [k]_t [k+1]_t \langle \alpha_s^\vee, x\rangle \alpha_t - [k+1]_s [k+1]_t \langle \alpha_t^\vee, x\rangle \alpha_t \\
& =x - ([k+2]_s [k]_t + 1)\langle \alpha_s^\vee, x\rangle \alpha_s - [k+2]_s [k+1]_s \langle \alpha_t^\vee, x\rangle \alpha_s \\
& \quad \quad {} - [k]_t [k+1]_t \langle \alpha_s^\vee, x\rangle \alpha_t - [k+1]_s [k+1]_t \langle \alpha_t^\vee, x\rangle \alpha_t \\
& =x - [k+1]_s [k+1]_t \langle \alpha_s^\vee, x\rangle \alpha_s - [k+2]_s [k+1]_s \langle \alpha_t^\vee, x\rangle \alpha_s \\
& \quad \quad {} - [k]_t [k+1]_t \langle \alpha_s^\vee, x\rangle \alpha_t - [k+1]_s [k+1]_t \langle \alpha_t^\vee, x\rangle \alpha_t \text{.}
\end{align*}

Since $(st)^m = (2m)_t$, the vanishing of $[m]_s$ and $[m]_t$ (a consequence of \eqref{eq:minpolycond} by \cite[Lemma~3.2]{2colJW}) implies that $(st)^m$ acts trivially on $V$. 
Thus $W$ acts on $V$ as prescribed.

To complete the proof, we note that for any $s \in S$, $x \in V$, and $y \in V^\ast$ we have
\begin{equation*}
\langle y, s(x) \rangle = \langle y, x \rangle - \langle y,\alpha_s \rangle \langle \alpha_s^\vee, x\rangle = \langle s(y), x\rangle \text{,}
\end{equation*}
so $W$ acts on $V^\ast$ as prescribed, in a dual manner to the action on $V$.
\end{proof}

\begin{cor}
Realizations as defined in Definition \ref{defn:rlz} are realizations in the sense of \cite[Definition~5.1]{2colJW}. 
In particular, a morphism of realizations of $(W,S)$ over $\Bbbk$ is always a homomorphism of $W$-representations.
\end{cor}


The most important consequence of this result is the construction of realizations directly from Cartan matrices. 

\begin{defn}
Let $A=(a_{st})_{s,t \in S}$ be a $\Bbbk$-valued Cartan matrix for a Coxeter system $(W,S)$.
\begin{itemize}
\item The \defnemph{adjoint realization} $(V_{A,{\rm ad}}, \{\alpha_s\}, \{\alpha_s^\vee\})$ of $(W, S)$ with respect to $A$ is defined as follows. 
Let $V_{A,{\rm ad}}$ be a free $\Bbbk$-module with basis $\{\alpha_s : s \in S\}$ and define $\{\alpha_s^\vee\} \subset V_{A,{\rm ad}}^\ast$ by
\begin{equation}
\langle \alpha_s^\vee, \alpha_t \rangle = a_{st} \qquad \text{ for all $s,t \in S$.} \label{eq:rootpairingcartan}
\end{equation}

\item The \defnemph{simply-connected realization} $(V_{A,{\rm sc}}, \{\alpha_s\}, \{\alpha_s^\vee\})$ of $(W, S)$ with respect to $A$ is defined as follows. 
Abusing notation let $V_{A,{\rm sc}}^\ast$ be a free $\Bbbk$-module with basis $\{\alpha_s^\vee : s \in S\}$ and define $V_{A,{\rm sc}}=(V_{A,{\rm sc}}^\ast)^\ast$ and $\{\alpha_s\} \subset V_{A,{\rm sc}}$ by \eqref{eq:rootpairingcartan} above.
\end{itemize}
\end{defn}

Clearly the adjoint and simply-connected realizations with respect to a given Cartan matrix are dual to each other. 
The adjoint realization was called the ``universal realization'' in \cite[Remark~1.5]{BHN-modularWeylKac}; the newer names more closely match similar terminology from the theory of root data.

\begin{lem} \label{lem:adscuniversality}
For any realization $V$ of $(W, S)$ over $\Bbbk$ with Cartan matrix $A=(a_{st})_{s,t \in S}$, there are unique morphisms of realizations 
\begin{equation*}
V_{A,{\rm ad}} \longrightarrow V \text{,} \qquad \qquad \qquad  V \longrightarrow V_{A,{\rm sc}} \text{.}
\end{equation*}
\end{lem}

We call the image of the first morphism the \defnemph{root subrealization} of $V$, and the coimage of the second morphism the \defnemph{coroot quotient realization}.

\begin{proof}
We prove the adjoint case only as the argument for the simply-connected case is dual.
To avoid notational confusion let $\{\beta_s\}$ and $\{\beta_s^\vee\}$ denote the roots and coroots of $V$ respectively.
By construction there is only one linear map $\phi: V_{A,{\rm ad}} \rightarrow V$ mapping roots to roots. 
Moreover for all $s,t \in S$ we have
\begin{equation*}
\langle \phi^\ast(\beta_s^\vee),\alpha_t \rangle=\langle \beta_s^\vee,\phi(\alpha_t)\rangle=\langle \beta_s^\vee,\beta_t\rangle=a_{st}=\langle \alpha_s^\vee, \alpha_t \rangle \text{.}
\end{equation*}
As $\alpha_s^\vee \in V_{A,{\rm ad}}^\ast$ is defined by this condition it follows that $\phi^\ast(\beta_s^\vee)=\alpha_s^\vee$.
\end{proof}

Realizations of $(W,S)$ over $\Bbbk$ naturally form a category, and later we will see that the construction of the Hecke category from a realization is functorial. 

\section{Unbalanced realizations without quantum numbers}

Recall that a realization of a Coxeter system $(W,S)$ is called \defnemph{balanced} if for all distinct $s,t \in S$ with $m_{st}<\infty$, the corresponding two-colored quantum numbers $[m_{st}-1]_s$ and $[m_{st}-1]_t$ are both $1$. 
The majority of the literature on diagrammatic Hecke categories solely treats balanced realizations, as they are notationally more straightforward to work with. 
To check for balancedness it is enough to check when $m_{st}$ is odd, as for even $m_{st}$ it can be shown that $[m_{st}-1]_s=[m_{st}-1]_t=1$ automatically \cite[Lemma~6.25]{ew-loccalc}.

In the general unbalanced case, the scalars $[m_{st}-1]_s$ and $[m_{st}-1]_t$ play a role in the diagrammatic relations defining the Hecke category. 
In this section we will give an alternative description of these scalars, using the minimal polynomials $\Psi'_m$ of $2\cos(\pi/m)$ over $\ZZ$. 
First we will show that a pair of distinct roots of $\Psi'_m$ never sum to $0$.

\begin{lem} \label{lem:psiprimebezout}
Suppose $m>1$ is odd. 
We have
\begin{equation}
\left(\prod_{1 \neq d|m-2} \Psi'_d(-x)\right)\left(\prod_{1 \neq d|m}\Psi'_d(x)\right)+\left(\prod_{1 \neq d|m-2} \Psi'_d(x)\right)\left(\prod_{1 \neq d|m}\Psi'_d(-x)\right) = \pm 2 \text{.} \label{eq:psiprimebezout}
\end{equation}
\end{lem}

\begin{proof}
Let $C_n=2T_n(x/2)$ be a family of rescaled Chebyshev polynomials of the first kind. 
When $n=2k+1$ is odd we have
\begin{equation*}
\prod_{d|n} \Psi'_d = \frac{C_{2k+2}-C_{2k}}{C_{k+1}-C_k}
\end{equation*}
from the main result of \cite{wz-minpolcos}. 
Since $\Psi'_1=x+2$ we obtain
\begin{equation}
\prod_{1 \neq d|n} \Psi'_d = \frac{C_{2k+2}-C_{2k}}{(x+2)(C_{k+1}-C_k)} \text{.} \label{eq:psiprimeprod}
\end{equation}

We will need a few identities involving the rescaled Chebyshev polynomials. 
First, we note that $C_k$ has the same parity as $k$, i.e.~$C_k(-x)=(-1)^k C_k$ for any integer $k \geq 0$.
Second, we will need the product identity $C_j C_k = C_{j+k} + C_{|j-k|}$ for any integers $j,k \geq 0$. 
As a consequence, for any integer $k \geq 0$ we have
\begin{align}
(C_{k}-C_{k-1})(C_{k+1}+C_k) - (C_{k}+C_{k-1})(C_{k+1}-C_k) & =2C_k^2 - 2C_{k-1}C_{k+1} \nonumber \\
& =2C_{2k}+2C_0-2C_{2k}-2C_2  \label{eq:chebprod1} \\
& =2(4-x^2) \nonumber
\end{align}
and
\begin{align}
(C_k-C_{k-1})(C_k+C_{k-1}) & =C_k^2 - C_{k-1}^2 \nonumber \\
& =C_{2k}+C_0-(C_{2k-2}+C_0) \label{eq:chebprod2}\\
& =C_{2k}-C_{2k-2} \text{.} \nonumber
\end{align}

Now suppose $m=2n+1$. 
Using \eqref{eq:psiprimeprod} we can simplify the left-hand side of \eqref{eq:psiprimebezout}:
\begin{multline*}
\left(\prod_{1 \neq d|m-2} \Psi'_d(-x)\right)\left(\prod_{1 \neq d|m}\Psi'_d(x)\right)+\left(\prod_{1 \neq d|m-2} \Psi'_d(x)\right)\left(\prod_{1 \neq d|m}\Psi'_d(-x)\right) \\
\begin{aligned}
& = \frac{(C_{2n}(-x)-C_{2n-2}(-x))(C_{2n+2}-C_{2n})}{(-x+2)(C_{n}(-x)-C_{n-1}(-x))(x+2)(C_{n+1}-C_n)} \\
& \quad \quad {} + \frac{(C_{2n}-C_{2n-2})(C_{2n+2}(-x)-C_{2n}(-x))}{(x+2)(C_{n}-C_{n-1})(-x+2)(C_{n+1}(-x)-C_n(-x))} \\
& =(-1)^n \frac{(C_{2n}-C_{2n-2})(C_{2n+2}-C_{2n})}{(4-x^2)(C_{n}+C_{n-1})(C_{n+1}-C_n)} \\
& \quad \quad {} -(-1)^n\frac{(C_{2n}-C_{2n-2})(C_{2n+2}-C_{2n})}{(4-x^2)(C_{n}-C_{n-1})(C_{n+1}+C_n)} \text{.}
\end{aligned}
\end{multline*}
Taking common denominators and applying \eqref{eq:chebprod1} and \eqref{eq:chebprod2} immediately gives the result:
\begin{multline*}
(-1)^n \frac{(C_{2n}-C_{2n-2})(C_{2n+2}-C_{2n})(C_{n}-C_{n-1})(C_{n+1}+C_n)}{(4-x^2)(C_{n}-C_{n-1})(C_n+C_{n-1})(C_{n+1}-C_n)(C_{n+1}+C_n)} \\
- (-1)^n\frac{(C_{2n}-C_{2n-2})(C_{2n+2}-C_{2n})(C_{n}+C_{n-1})(C_{n+1}-C_n)}{(4-x^2)(C_{n}-C_{n-1})(C_n+C_{n-1})(C_{n+1}-C_n)(C_{n+1}+C_n)} \\
\begin{aligned}
& =(-1)^n\left(\frac{2(C_{2n}-C_{2n-2})(C_{2n+2}-C_{2n})(4-x^2)}{(4-x^2)(C_{2n}-C_{2n-2})(C_{2n+2}-C_{2n})}\right) \\
& =2(-1)^n \qedhere
\end{aligned}
\end{multline*} 
\end{proof}

\begin{cor} \label{cor:psiprimerootsum}
Let $\Bbbk$ be a commutative integral domain, and suppose $m$ is a positive integer. 
If $r \in \Bbbk$ is a root of $\Psi'_m$, then $-r$ is a root as well if and only if $r=-r$.
\end{cor}

\begin{proof}
Lemma \ref{lem:psiprimebezout} gives a linear combination of $\Psi'_m(x)$ and $\Psi'_m(-x)$ which is equal to $\pm 2$. 
If both $r$ and $-r$ are roots of $\Psi'_m$, then they are roots of \eqref{eq:psiprimebezout} as well, so $\pm 2=0$ in $\Bbbk$, and therefore $r=-r$.
\end{proof}

\begin{thm} \label{thm:mminus1twocol}
Let $V$ be a realization of $(W,S)$ over $\Bbbk$.
Suppose $m=m_{st}>1$ is odd.
Then $\langle \alpha_s^\vee, \alpha_t \rangle \langle \alpha_t^\vee, \alpha_s\rangle$ is invertible and has a square root in $\Bbbk$. 
Moreover, if we fix the sign of this square root so that $\Psi'_m(\sqrt{\langle \alpha_s^\vee, \alpha_t\rangle \langle \alpha_t^\vee, \alpha_s\rangle})=0$ using Corollary \ref{cor:psiprimerootsum} then 
\begin{equation}
[m-1]_s = -\frac{\langle \alpha_s^\vee,\alpha_t\rangle}{\sqrt{\langle \alpha_s^\vee, \alpha_t\rangle \langle \alpha_t^\vee, \alpha_s\rangle}} \text{.} \label{eq:mminus1sqrt}
\end{equation}
\end{thm}

\begin{rem}
Since $[m]_s=[m]_t=0$ from the definition of a realization, it is not hard to show that $[m-1]_s[m-1]_t=1$ (see e.g.~\cite[(6.11)]{ew-loccalc}).
Using this and \eqref{eq:2colqnumfrom1col} below it follows that
\begin{equation*}
[m-1]_s^2=\frac{\langle \alpha_s^\vee, \alpha_t\rangle}{\langle \alpha_t^\vee, \alpha_s \rangle} \text{.}
\end{equation*}
In other words, the real content of the theorem is in the determination of the \emph{sign} of $[m-1]_s$.
\end{rem}

\begin{proof}
Recall the (generic) one-colored quantum numbers, defined as univariate polynomials in a variable $x$ by setting $[0]=0$, $[1]=1$, $[2]=x$, and inductively defining $[2][n]=[n+1]+[n-1]$.
These are essentially Chebyshev polynomials of the second kind (e.g.~\cite[\S 3]{2colJW}), i.e.
\begin{equation*}
[n](2\cos \theta)=U_{n-1}(\cos \theta)=\frac{\sin n\theta}{\sin \theta} \text{.} 
\end{equation*}
The parity of $[n]$ (as a polynomial) is the opposite of the parity of $n$. 
In particular, when $n$ is even, $[2]$ divides $[n]$.

We also recall the \defnemph{cyclotomic parts}
\begin{equation*}
\Theta_n = \prod_{\substack{1 \leq k<n\\ (k,n)=1}} \left(x-2\cos \frac{k\pi}{n}\right)
\end{equation*}
from \cite[\S 3]{2colJW}, for which $\Theta_n(x)=\Psi_n(x^2)$ when $n>2$. 
For odd $n>1$ it is not hard to see that $\Psi'_n(x)\Psi'_n(-x)=\Theta_n$, as every root $2\cos(k\pi/n)$ of $\Theta_n$ is either a root of $\Psi'_n(x)$ or of $\Psi'_n(-x)$ depending on whether $k$ is odd or even.

The polynomial $[m-1]-1$ has $x=2\cos(\pi/m)$ as a root. 
Therefore $\Psi'_m$ divides $[m-1]-1$, and similarly $\Psi'_m(-x)$ divides $[m-1]+1$. 
This means that $\Psi'_m(x)\Psi'_m(-x)=\Psi_m(x^2)$ divides $[m-1]^2-1$. 
Moreover, as both polynomials are even, their quotient is even too.
Thus $u=\frac{[m-1]}{[2]} \in \ZZ[x^2]/(\Psi_m(x^2))$ is a root of the polynomial equation
\begin{equation}
u^2 x^2 - 1 \in (\ZZ[x^2]/(\Psi_m(x^2)))[u] \text{.} \label{eq:invsqroot}
\end{equation}
We can determine the sign of this root by observing that $[m-1][2]=[m-1]x$ is always a root of $\Psi'_m$ modulo $\Psi_m(x^2)$.

Now we use the relationship between one-colored and two-colored quantum numbers.
It can be shown (see e.g.~\cite[(6)]{2colJW}) that
\begin{equation}
\begin{aligned}
[n]_{s}& =[n](\sqrt{x_{s} x_{t}})=[n]_{t} & & \text{if $n$ is odd,} \\
\frac{[n]_{s}}{[2]_{s}}& =\left(\frac{[n]}{[2]}\right)(\sqrt{x_{s} x_{t}})=\frac{[n]_{t}}{[2]_{t}} & & \text{if $n$ is even.}
\end{aligned} \label{eq:2colqnumfrom1col}
\end{equation}
in $A=\ZZ[x_s,x_t]$.
Working in the quotient ring $B=\ZZ[x_s,x_t]/(\Psi_{m}(x_s x_t))$ it follows from \eqref{eq:invsqroot} that
\begin{equation*}
u=\frac{[m-1]_s}{[2]_s}=\frac{[m-1]}{[2]}(\sqrt{x_s x_t})
\end{equation*}
is a root of the equation $u^2 x_s x_t - 1 \in B[t]$.
In other words $x_s x_t$ has an invertible square root in $B$ which is a root of $\Psi'_m$; its reciprocal is equal to $\frac{[m-1]_s}{[2]_s}$.
We complete the proof by substituting $x_s=-\langle \alpha_s^\vee, \alpha_t\rangle$ and $x_t=-\langle \alpha_t^\vee, \alpha_s\rangle$. 
\end{proof}

\begin{cor} \label{cor:balancedcriterion}
Let $V$ be a realization of $(W,S)$ over $\Bbbk$.
Then $V$ is balanced if and only if for all distinct $s,t \in S$ such that $m_{st}$ is odd, we have $\langle \alpha_s^\vee,\alpha_t\rangle=\langle \alpha_t^\vee,\alpha_s\rangle$ and $\Psi'_{m_{st}}(-\langle \alpha_s^\vee, \alpha_t \rangle)=\Psi'_{m_{st}}(-\langle \alpha_t^\vee,\alpha_s\rangle)=0$.
\end{cor}

\section{Demazure surjectivity and cinched Hecke categories}

An extremely common assumption on realizations in the literature is Demazure surjectivity, which we recall (and slightly generalize) below.

\begin{defn} \label{defn:demazuresurj}
Let $V$ be a realization of $(W,S)$ over $\Bbbk$.
\begin{itemize}
\item We say $V$ is \defnemph{coroot surjective} if $\alpha_s^\vee : V \rightarrow \Bbbk$ is surjective for all $s \in S$.

\item We say $V$ is \defnemph{root surjective} if $\alpha_s : V^\ast \rightarrow \Bbbk$ is surjective for all $s \in S$.

\item We say $V$ is \defnemph{Demazure surjective} if it is both coroot surjective and root surjective.
\end{itemize}
\end{defn}

Demazure surjectivity is often conflated with coroot surjectivity in the literature. 
For example, it is sometimes claimed that Demazure surjectivity is \emph{necessary} for the Demazure operators $\{\partial_s\}$ to be surjective, or that every realization embeds inside a Demazure surjective realization (e.g.~\cite[\S 2.3]{gjw-calcpcanbasis}). 
These are both incorrect as written, but the corresponding statements for coroot surjective realizations are true.

%
%

\begin{exam} \label{exam:rootcorootsurj}
Let $\Bbbk=\mathbb{F}_2$ and $(W,S)$ be the $A_1$ Coxeter group. 
We consider two different realization structures on the one-dimensional vector space $V=\Bbbk \varepsilon$. 
Write $\varepsilon^\ast \in V^\ast$ for the dual basis vector, i.e.~$\varepsilon^\ast(\varepsilon)=1$.
\begin{enumerate}[ref=(\arabic*)]
\item \label{item:corootsurj} Let $\alpha_s=0 \in V$ and $\alpha_s^\vee=\varepsilon^\ast \in V^\ast$. 
Clearly $V$ is coroot surjective but not root surjective, so $V$ is not Demazure surjective.
We also observe that $\partial_s(\varepsilon)=\alpha_s^\vee(\varepsilon)=1$, so the Demazure operator $\partial_s$ is surjective.

\item \label{item:zerorlz} Let $\alpha_s=0 \in V$ and $\alpha_s^\vee=0 \in V^\ast$. 
Now $V$ is neither coroot surjective nor root surjective. 
If $V'$ is another realization and $\phi:V \rightarrow V'$ is a morphism of realizations, then $\alpha_s'=\phi(\alpha_s)=\phi(0)=0$. 
In particular $V'$ cannot be Demazure surjective. 
\end{enumerate}
\end{exam}

Let $V$ be a $\Bbbk$-realization of $(W,S)$, and let $R=\Sym(V)$ be the symmetric algebra of $V$ with $\deg V=2$. 
For the moment let us assume that the roots of $V$ are non-zero, and for each $s \in S$ write $R^s$ for the subring of invariants with respect to $s$. 
We will say a little more about the relationship between the various surjectivity conditions and different incarnations of the Hecke category. 

In the classical (i.e.~non-diagrammatic) theory of Soergel bimodules, surjectivity of the Demazure operator $\partial_s$ gives $R^s \subset R$ the structure of a Frobenius extension, which can be used to show that the Bott--Samelson bimodule $B_s=R \otimes_{R^s} R(1)$ is a Frobenius algebra object (see e.g.~\cite[Lemma~8.27]{emtw}).
In the diagrammatic approach the corresponding object $B_s$ is automatically given a Frobenius algebra object structure by diagrammatic relations. 
Here coroot surjectivity is only necessary to show that the tensor square $B_s^{\otimes 2}$ splits as a direct sum $B_s(1) \oplus B_s(-1)$ (see e.g.~\cite[(5.14)]{EW-SoergelCalc}), i.e.~the simplest case of Soergel's categorification theorem \cite[Theorem~11.1]{emtw}.

By contrast, root surjectivity on its own is rather strong. 
Our simpler underlying assumption of non-zero roots is already enough to ensure that the localization functor is well defined \cite{ew-loccalc}.
This combined with coroot surjectivity implies that Soergel's categorification theorem holds.
Demazure surjectivity is a stronger assumption, but it is convenient because it ensures that Soergel's categorification theorem also holds for the \emph{dual} realization.

\begin{rem}
We have (deliberately) elided what the codomain of $\partial_s$ should be whenever we say ``$\partial_s$ is surjective''. 
Let $Z^s=\ker \partial_s$. 
It can be shown that $\partial_s^2=0$, so the Demazure operator defines a ($Z^s$-bimodule) map $\partial_s : R \rightarrow Z^s(-2)$. 
Moreover, since $s \circ \partial_s = \partial_s$, we have $Z^s \subset R^s$.
The standard argument actually shows that if $\partial_s$ is surjective as a map $R \rightarrow Z^s(-2)$, then $Z^s \subset R$ is a Frobenius extension. 
(If $\alpha_s \neq 0$ as well then $Z^s=R^s$, which gives us the results above.)

So if $\partial_s$ is surjective onto $Z^s$, then $R \otimes_{Z^s} R(1)$ is a Frobenius algebra object, regardless of whether or not $\alpha_s$ is non-zero. 
We believe that this bimodule should behave like a genuine Bott--Samelson generator, with $\Hom$-spaces of the correct dimension.
This would potentially allow us to deal with realizations which are only coroot surjective (such as in Example \ref{exam:rootcorootsurj}\ref{item:corootsurj}) using bimodules. 
The diagrammatic approach we take below is more general, but at the cost of any functor to a category of bimodules.
\end{rem}

\begin{rem}
One notable situation where the roots being surjective (as opposed to just non-zero) matters arises from the left monodromy action on the relevant \defnemph{mixed derived categories} \cite[\S 2.3]{amrw-koszulduality}. 
It can be shown that the monodromy maps for the generating tilting objects are surjective if and only if the realization is root surjective \cite[Example~4.7.4]{amrw-freemonodromic}. 
As this action is closely connected to Koszul duality for Hecke categories (which exchanges the underlying realization with its dual), it is perhaps unsurprising that root surjectivity is meaningful here. 
\end{rem}



In fact \emph{none} of the surjectivity conditions in Definition \ref{defn:demazuresurj} are necessary to construct a diagrammatic categorification of the Hecke algebra. 
We carefully explain how to do this below. 
We emphasize again that in such generality, one should not expect there to be any functor to a category of bimodules.

Let $\mathcal{D}_{\rm BS}$ denote the Elias--Williamson diagrammatic category for an arbitrary $\Bbbk$-realization $V$ of a Coxeter system $(W,S)$ (introduced in \cite[Definition 5.2, (5.1)--(5.12)]{EW-SoergelCalc}). 
In particular, we will assume the following version of the polynomial forcing relation 
\begin{equation}
\begin{gathered}
    \begin{tikzpicture}[scale=0.7,baseline=(poly)]
        \node[circle, minimum size=0.66cm] (spacer) at (0,0) {};
        \draw[densely dotted, rounded corners] (-1.25,-1) rectangle (1.25,1) ;
        \coordinate (above) at (0,1);
        \coordinate (below) at (0,-1);
        
        \draw[red, string] (above) to node[anchor=east, name=poly] [black] {$v$} (below);
    \end{tikzpicture}
\end{gathered}
-
\begin{gathered}
    \begin{tikzpicture}[scale=0.7,baseline=(poly)]
        \node[circle, minimum size=0.5cm] (spacer) at (0,0) {};
        \draw[densely dotted, rounded corners] (-1.25,-1) rectangle (1.25,1) ;
        \coordinate (above) at (0,1);
        \coordinate (below) at (0,-1);
        
        \draw[red, string] (above) to node[anchor=west, name=poly] [black] {${\color{red} s}(v)$} (below);
    \end{tikzpicture}
\end{gathered}
=
\begin{gathered}
    \begin{tikzpicture}[scale=0.7,baseline=(poly)]
        \node[circle, minimum size=0.5cm] (spacer) at (0,0) {};
        \draw[densely dotted, rounded corners] (-1,-1) rectangle (1,1) ;
        \coordinate (above) at (0,1);
        \coordinate (below) at (0,-1);
        \node[dot, fill=red] (updot) at (0,0.5) {};
        \node[dot, fill=red] (downdot) at (0,-0.5) {};
        
        \path (above) to node[anchor=center, name=poly] {$\langle \alpha^\vee_{\color{red} s}, v\rangle$} (below);
        \draw[red, string] (above) to (updot.center)
                           (below) to (downdot.center);
    \end{tikzpicture}
\end{gathered}\qquad \qquad \text{for all $s \in S$ and $v \in V$,} \label{eq:pushpoly}
\end{equation}
which is equivalent to \cite[(5.2)]{EW-SoergelCalc} when $\alpha_s \neq 0$ but still makes sense otherwise. 
This is the only relation which depends on the realization and not just the Cartan matrix. 
We warn the reader that in this setting, the statements
\begin{gather}
s(f)=f \quad \Longleftrightarrow \quad \partial_{s}(f)=0 \qquad \text{for all $f \in R$,} \\
\begin{gathered}
    \begin{tikzpicture}[scale=0.7,baseline=(poly)]
        \node[circle, minimum size=0.5cm] (spacer) at (0,0) {};
        \draw[densely dotted, rounded corners] (-0.66,-1) rectangle (0.66,1) ;
        \coordinate (above) at (0,1);
        \coordinate (below) at (0,-1);        
        \draw[red, string] (above) to node[anchor=east, name=poly] [black] {$f$} (below);
    \end{tikzpicture}
\end{gathered}
=
\begin{gathered}
    \begin{tikzpicture}[scale=0.7,baseline=(poly)]
        \node[circle, minimum size=0.5cm] (spacer) at (0,0) {};
        \draw[densely dotted, rounded corners] (-0.66,-1) rectangle (0.66,1) ;
        \coordinate (above) at (0,1);
        \coordinate (below) at (0,-1);        
        \draw[red, string] (above) to node[anchor=west, name=poly] [black] {$f$} (below);
    \end{tikzpicture}
\end{gathered} \qquad \text{for all $f \in R^{\color{red} s}$}
\end{gather}
do \emph{not} hold in general (cf.~Example \ref{exam:rootcorootsurj}\ref{item:corootsurj}).
Moreover, we cannot use the theory of localization described in \cite{EW-SoergelCalc,ew-loccalc} without extra work (see Claim \ref{clm:LLlinindep} below).

It turns out that a quotient of $\mathcal{D}_{\rm BS}$ \emph{always} categorifies the Hecke algebra, and thus deserves the moniker ``Hecke category''. 
(This approach was first outlined without proof in \cite[Remark~1.15]{bdhn-hsp}.) 
In what follows we will \emph{not} call $\mathcal{D}_{\rm BS}$ ``the diagrammatic Bott--Samelson category'' or its Karoubi envelope $\mathcal{D}$ ``the diagrammatic Hecke category'', to avoid potential confusion.
The proposed replacement for $\mathcal{D}_{\rm BS}$ involves adding a single extra diagrammatic relation.

\begin{defn}
The \defnemph{(cinched) diagrammatic Bott--Samelson category} $\mathcal{H}_{\rm BS}$ is the quotient of $\mathcal{D}_{\rm BS}$ with respect to the additional diagrammatic (i.e.~monoidal) ``cinching relation'' 
\begin{equation}
\begin{minipage}{1.5cm}\begin{tikzpicture}[scale=1.000]
\draw[densely dotted, rounded corners] (-0.5,0) rectangle (1,1.5) ;
\draw[red,line width=0.08cm](0.5,0)--++(90:1.5);
\draw[red,line width=0.08cm](0,0)--++(90:1.5);   \end{tikzpicture}
  \end{minipage}
  \;=\;
   \begin{minipage}{1.5cm}\begin{tikzpicture}[scale=1.000]
\draw[densely dotted, rounded corners] (-0.5,0) rectangle (1,1.5) ;
%
\draw[red,line width=0.08cm](0,1.5) to [out=-90,in=180]
(0.5,1.5-0.45)   ;

  \draw[red,line width=0.08cm](0.5,0)--++(90:1.5);

\draw[red,line width=0.08cm] (0 ,0)--++(90:0.45) coordinate(hi);
\fill[red] (hi) circle (3.5pt);
  \end{tikzpicture}
  \end{minipage}
\;  + \;  \begin{minipage}{1.5cm}\begin{tikzpicture}[scale=1.000,yscale=-1.000]
\draw[densely dotted, rounded corners] (-0.5,0) rectangle (1,1.5) ;
%
\draw[red,line width=0.08cm](0,1.5) to [out=-90,in=180]
(0.5,1.5-0.45)   ;

  \draw[red,line width=0.08cm](0.5,0)--++(90:1.5);

\draw[red,line width=0.08cm] (0 ,0)--++(90:0.45) coordinate(hi);
\fill[red] (hi) circle (3.5pt);
  \end{tikzpicture}
  \end{minipage}
\; - 
 \;
 \begin{minipage}{1.5cm}\begin{tikzpicture}[scale=1.000]
\draw[densely dotted, rounded corners] (-0.5,0) rectangle (1,1.5) ;
%
\draw[red,line width=0.08cm](0,1.5) to [out=-90,in=180]
(0.5,1.5-0.45)   ;
\draw[red,line width=0.08cm](0,0) to [out=90,in=180]
(0.5,0.45)   ;

\draw[red,line width=0.08cm](0.5,0)--++(90:1.5);
\draw[red,line width=0.08cm] (0-0.2 ,0.75)--++(90:0.3) coordinate(hi);
\fill[red] (hi) circle (3.5pt);
\draw[red,line width=0.08cm] (0-0.2 ,0.75)--++(-90:0.3) coordinate(hi);
\fill[red] (hi) circle (3.5pt);
  \end{tikzpicture}
  \end{minipage} \label{eq:cinching}
\end{equation}
for each color.
We call the Karoubi envelope of the additive graded closure of $\mathcal{H}_{\rm BS}$ the \defnemph{(cinched) diagrammatic Hecke category} $\mathcal{H}$.
\end{defn}

The cinching relation \eqref{eq:cinching} first appeared as an independent relation in \cite[(R4)]{bdhn-hsp} along with informal versions of many of the results below. 
Adding a dot to the right strand gives an equivalent form of the cinching relation:
\begin{equation}
\begin{minipage}{1.5cm}
\begin{tikzpicture}[scale=1.000]
\draw[densely dotted, rounded corners] (-0.5,0) rectangle (1,1.5) ;
\draw[red,line width=0.08cm](0.5,0)--++(90:0.45) coordinate (dot);
\fill[red] (dot) circle (3.5pt);
\draw[red,line width=0.08cm](0,0)--++(90:1.5);   \end{tikzpicture}
  \end{minipage}
  \;=\;
   \begin{minipage}{1.5cm}\begin{tikzpicture}[scale=1.000]
\draw[densely dotted, rounded corners] (-0.5,0) rectangle (1,1.5) ;
%

  \draw[red,line width=0.08cm](0.5,0)--++(90:1.5);

\draw[red,line width=0.08cm] (0 ,0)--++(90:0.45) coordinate(hi);
\fill[red] (hi) circle (3.5pt);
  \end{tikzpicture}
  \end{minipage}
\;  + \;  \begin{minipage}{1.5cm}\begin{tikzpicture}[scale=1.000,yscale=-1.000]
\draw[densely dotted, rounded corners] (-0.5,0) rectangle (1,1.5) ;
%
\draw[red,line width=0.08cm](0,1.5) arc (180:360:0.25);


\draw[red,line width=0.08cm] (0.25,0)--++(90:0.45) coordinate(hi);
\fill[red] (hi) circle (3.5pt);
  \end{tikzpicture}
  \end{minipage}
\; - 
 \;
 \begin{minipage}{1.5cm}\begin{tikzpicture}[scale=1.000]
\draw[densely dotted, rounded corners] (-0.5,0) rectangle (1,1.5) ;
%
\draw[red,line width=0.08cm](0,0) to [out=90,in=180]
(0.5,0.45)   ;

\draw[red,line width=0.08cm](0.5,0)--++(90:1.5);
\draw[red,line width=0.08cm] (0-0.2 ,0.75)--++(90:0.3) coordinate(hi);
\fill[red] (hi) circle (3.5pt);
\draw[red,line width=0.08cm] (0-0.2 ,0.75)--++(-90:0.3) coordinate(hi);
\fill[red] (hi) circle (3.5pt);
  \end{tikzpicture}
  \end{minipage} \label{eq:dottedcinch}
\end{equation}

\begin{rem} \hfill
\begin{enumerate}
\item We have included the term ``cinched'' here only to avoid potential confusion with the original Elias--Williamson construction. 
Once we have shown that $\mathcal{H}_{\rm BS}$ is better behaved, and in particular \emph{always} categorifies the Hecke algebra, it is reasonable to redefine the terms ``diagrammatic Bott--Samelson category'' and ``diagrammatic Hecke category'' so that they are always ``cinched'' (i.e.~always include the cinching relation) (cf.~\cite{bowman-book}).

\item Many results for $\mathcal{D}_{\rm BS}$ carry over to $\mathcal{H}_{\rm BS}$ without change. 
For example, it is easy to see that the construction of the cinched Bott--Samelson category is functorial in the realization, and that localization is still well defined whenever the roots of the underlying realization are all non-zero. 
\end{enumerate}
\end{rem}

Recall the \defnemph{light leaves} and \defnemph{double leaves} constructions for $\mathcal{D}_{\rm BS}$ \cite[\S 6.1, \S 6.3]{EW-SoergelCalc}. 
These diagrammatic constructions carry over identically to the cinched Hecke category. 
The main result in this section is the following. 

\begin{thm}[{cf.~\cite[Theorem 6.11, Proposition 7.6]{EW-SoergelCalc}}] \label{thm:LLcinched} \hfill
%

The light leaves and double leaves constructions yield bases for the $\Hom$-spaces of the cinched Bott--Samelson category $\mathcal{H}_{\rm BS}$.
\end{thm}

As in $\mathcal{D}_{\rm BS}$ this immediately implies the following.

\begin{cor}[{cf.~\cite[Theorem 11.1, Assumption 11.25]{emtw}}] \label{cor:cinchedSoergelcatthm} \hfill

Suppose $\Bbbk$ is a henselian\footnote{E.g.~a field, or a complete local ring.} local ring. 
Then Soergel's categorification theorem holds for the cinched Hecke category $\mathcal{H}$.
\end{cor}

\begin{proof}[Proof of Theorem \ref{thm:LLcinched}]
We follow the same strategy as in the proof of \cite[Proposition~7.6]{EW-SoergelCalc}, by first showing that the light leaves maps are linearly independent, and then showing that they span the relevant $\Hom$-space modulo lower terms. 

If we assume that the roots in $V$ are all non-zero then the usual localization argument carries over to $\mathcal{H}_{\rm BS}$ to prove linear independence. 
For completeness we will show that this assumption is unnecessary.

\begin{clm} \label{clm:LLlinindep}
The light leaves maps in $\mathcal{H}_{\rm BS}$ are linearly independent.
\end{clm}

\begin{proof}
Let $I=\{t \in S : \alpha_t=0\}$. 
In $V$ we observe that $\langle \alpha^\vee_s,\alpha_t \rangle=0$ whenever $t \in I$. 
Define a new realization
\begin{equation*}
V'=V \oplus \bigoplus_{s \in I} \Bbbk \varepsilon_s
\end{equation*}
by setting
\begin{align*}
\alpha'_s& =\alpha_s & &\text{if $s \notin I$,} \\
\alpha'_s& =\varepsilon_s & &\text{if $s \in I$,} \\
\langle (\alpha'_s)^\vee, v \rangle& =\langle \alpha_s^\vee, v\rangle & &  \text{if  $s \in S$ and $v \in V$,} \\
\langle (\alpha'_s)^\vee, \varepsilon_t\rangle & =0 & &  \text{if $s \in S$ and $t \in I$.}
\end{align*}
Let $\mathcal{H}'_{\rm BS}$ denote the cinched Bott--Samelson category constructed from $V'$, and let $K=\bigoplus_{s \in I} \Bbbk \alpha'_s \leq V'$.
Due to the form of the diagrammatic relations, it is evident that $\mathcal{H}_{\rm BS}$ is the quotient of $\mathcal{H}'_{\rm BS}$ modulo any Soergel diagram with a term in $K$ in some region. 
Any such term in $K$ can be pushed all the way to the left using the polynomial forcing relation \eqref{eq:pushpoly} without introducing any extra error terms.
This shows that $\mathcal{H}_{\rm BS} = R \otimes_{R'} \mathcal{H}_{\rm BS}$, where $R$ and $R'$ denote the symmetric algebras of $V$ and $V'$ respectively.  
Now observe that the light leaves maps in $\mathcal{H}'_{\rm BS}$ are linearly independent by the usual localization argument, as the roots of $V'$ are all non-zero.
In each relevant $\Hom$-space, the light leaves maps for $\mathcal{H}'_{\rm BS}$ span a free left $R'$-submodule. 
Changing scalars to $R$ preserves freeness, so the corresponding light leaves maps are linearly independent in $\mathcal{H}_{\rm BS}$ too.
\end{proof}

To show that these maps also span the relevant $\Hom$-space modulo lower terms, the proof is identical to the usual argument in \cite[\S 7]{EW-SoergelCalc}, noting that the only use of coroot surjectivity there is in the proof of \cite[Claim~7.10]{EW-SoergelCalc}. 
In $\mathcal{H}_{\rm BS}$ it is enough to prove a slightly weaker version of this statement. 
See \cite[\S\S 6--7]{EW-SoergelCalc} for a description of the notation (and the definition of ``lower terms'').

\begin{clm}
Let $\underline{x}$ be an expression in $S$, and let $\mathbf{e}$ be a subsequence of $\underline{x}$ which expresses some $w \in W$.
Suppose $s \in S$ such that $ws<w$. 
Then 
\begin{equation*}
\mathrm{LL}_{\underline{x},\mathbf{e}} \otimes {\rm dot}_s = 
\begin{tikzpicture}[scale=0.4,baseline=(std)] 
\def\strings{6}
\pgfmathsetmacro{\stringgap}{1.0/\strings}
\node (std) at (0,0) [draw, rectangle, minimum width=2.4cm, minimum height=0.6cm] {$\mathrm{LL}_{\underline{x},\mathbf{e}}$};
\foreach \i in {1,...,\strings}
{
    \pgfmathsetmacro{\xoffset}{(\i-0.5)*\stringgap}
    \path ($(std.north west)!\xoffset!(std.north east)$) node (topb\i) {};
    \path (topb\i) ++(0,1.5cm) node (topa\i) {};
    \path ($(std.south west)!\xoffset!(std.south east)$) node (bota\i) {};
    \path (bota\i) 
                   ++(0,-1.5cm) node (botc\i) {};
}
\pgfmathsetmacro{\dotoffset}{(\strings+0.75)*\stringgap}
\path ($(std.west)!\dotoffset!(std.east)$) node[dot,fill=red] (dot) {};
\path (dot.center |- botc6.center) node (botdot) {};
\foreach \i in {1,...,6}
{
    \pgfmathsetmacro{\xoffset}{(\i-0.5)*\stringgap}
    \draw (bota\i.center) edge[string] (botc\i.center);
}
\draw (bota3.center) edge[string,red] (botc3.center);
\draw (dot.center) edge[string,red] (botdot.center);
\foreach \i in {2,...,5}
{
    \pgfmathsetmacro{\xoffset}{(\i-0.5)*\stringgap}
    \draw (topa\i.center) edge[string] (topb\i.center);
}
\draw (topa5.center) edge[string,red] (topb5.center);
\end{tikzpicture}
\end{equation*}
can be written as a linear combination of light leaves maps with polynomials in any region modulo lower terms.
\end{clm}

\begin{proof}
Apply the modified form of the cinching relation \eqref{eq:dottedcinch}:
\begin{equation*}
\begin{tikzpicture}[scale=0.4,baseline=(std)] 
\def\strings{6}
\pgfmathsetmacro{\stringgap}{1.0/\strings}
\node (std) at (0,0) [draw, rectangle, minimum width=2.4cm, minimum height=0.6cm] {$\mathrm{LL}_{\underline{x},\mathbf{e}}$};
\foreach \i in {1,...,\strings}
{
    \pgfmathsetmacro{\xoffset}{(\i-0.5)*\stringgap}
    \path ($(std.north west)!\xoffset!(std.north east)$) node (topb\i) {};
    \path (topb\i) ++(0,3.5cm) node (topa\i) {};
    \path ($(std.south west)!\xoffset!(std.south east)$) node (bota\i) {};
    \path (bota\i) 
                   ++(0,-1.5cm) node (botc\i) {};
}
\pgfmathsetmacro{\dotoffset}{(\strings+0.75)*\stringgap}
\path ($(std.west)!\dotoffset!(std.east)$) node[dot,fill=red] (dot) {};
\path (dot.center |- botc6.center) node (botdot) {};
\foreach \i in {1,...,6}
{
    \pgfmathsetmacro{\xoffset}{(\i-0.5)*\stringgap}
    \draw (bota\i.center) edge[string] (botc\i.center);
}
\draw (bota3.center) edge[string,red] (botc3.center);
\draw (dot.center) edge[string,red] (botdot.center);
\foreach \i in {2,...,5}
{
    \pgfmathsetmacro{\xoffset}{(\i-0.5)*\stringgap}
    \draw (topa\i.center) edge[string] (topb\i.center);
}
\draw (topa5.center) edge[string,red] (topb5.center);
\end{tikzpicture}
=
\begin{tikzpicture}[scale=0.4,baseline=(std)] 
\def\strings{6}
\pgfmathsetmacro{\stringgap}{1.0/\strings}
\node (std) at (0,0) [draw, rectangle, minimum width=2.4cm, minimum height=0.6cm] {$\mathrm{LL}_{\underline{x},\mathbf{e}}$};
\foreach \i in {1,...,\strings}
{
    \pgfmathsetmacro{\xoffset}{(\i-0.5)*\stringgap}
    \path ($(std.north west)!\xoffset!(std.north east)$) node (topb\i) {};
    \path (topb\i) ++(0,3.5cm) node (topa\i) {};
    \path ($(std.south west)!\xoffset!(std.south east)$) node (bota\i) {};
    \path (bota\i) 
                   ++(0,-1.5cm) node (botc\i) {};
}
\pgfmathsetmacro{\dotoffset}{(\strings+0.75)*\stringgap}
\path ($(std.west)!\dotoffset!(std.east)$) node (dot) {};
\path (dot.center |- botc6.center) node (botdot) {};
\path ($(topa5.center)!0.25!(topb5.center)$) node (turn1) {};
\path (turn1.center -| botdot.center) node (turn2) {};
\path ($(turn1.center)!0.5!(turn2.center)$) node (turn) {};
\path ($(topa5.center)!0.75!(topb5.center)$) node[dot,fill=red] (dotnew) {};
\foreach \i in {1,...,6}
{
    \pgfmathsetmacro{\xoffset}{(\i-0.5)*\stringgap}
    \draw (bota\i.center) edge[string] (botc\i.center);
}
\draw (bota3.center) edge[string,red] (botc3.center);
\draw (dot.center) edge[string,red] (botdot.center);
\foreach \i in {2,...,4}
{
    \pgfmathsetmacro{\xoffset}{(\i-0.5)*\stringgap}
    \draw (topa\i.center) edge[string] (topb\i.center);
}
\draw[corner, string, red] (topa5.center) |- (turn.center) -| (botdot.center);
\draw[string,red] (dotnew.center) edge (topb5.center);
\end{tikzpicture}
+
\begin{tikzpicture}[scale=0.4,baseline=(std)] 
\def\strings{6}
\pgfmathsetmacro{\stringgap}{1.0/\strings}
\node (std) at (0,0) [draw, rectangle, minimum width=2.4cm, minimum height=0.6cm] {$\mathrm{LL}_{\underline{x},\mathbf{e}}$};
\foreach \i in {1,...,\strings}
{
    \pgfmathsetmacro{\xoffset}{(\i-0.5)*\stringgap}
    \path ($(std.north west)!\xoffset!(std.north east)$) node (topb\i) {};
    \path (topb\i) ++(0,3.5cm) node (topa\i) {};
    \path ($(std.south west)!\xoffset!(std.south east)$) node (bota\i) {};
    \path (bota\i) 
                   ++(0,-1.5cm) node (botc\i) {};
}
\pgfmathsetmacro{\dotoffset}{(\strings+0.75)*\stringgap}
\path ($(std.west)!\dotoffset!(std.east)$) node (dot) {};
\path (dot.center |- botc6.center) node (botdot) {};
\path ($(topa5.center)!0.25!(topb5.center)$) node[dot,fill=red] (dotnew) {};
\path ($(topa5.center)!0.75!(topb5.center)$) node (turn1) {};
\path (turn1.center -| botdot.center) node (turn2) {};
\path ($(turn1.center)!0.5!(turn2.center)$) node (turn) {};
\foreach \i in {1,...,6}
{
    \pgfmathsetmacro{\xoffset}{(\i-0.5)*\stringgap}
    \draw (bota\i.center) edge[string] (botc\i.center);
}
\draw (bota3.center) edge[string,red] (botc3.center);
\foreach \i in {2,...,4}
{
    \pgfmathsetmacro{\xoffset}{(\i-0.5)*\stringgap}
    \draw (topa\i.center) edge[string] (topb\i.center);
}
\draw[corner, string, red] (topb5.center) |- (turn.center) -| (botdot.center);
\draw[string,red] (dotnew.center) edge (topa5.center);
\end{tikzpicture}
-
\begin{tikzpicture}[scale=0.4,baseline=(std)] 
\def\strings{6}
\pgfmathsetmacro{\stringgap}{1.0/\strings}
\node (std) at (0,0) [draw, rectangle, minimum width=2.4cm, minimum height=0.6cm] {$\mathrm{LL}_{\underline{x},\mathbf{e}}$};
\foreach \i in {1,...,\strings}
{
    \pgfmathsetmacro{\xoffset}{(\i-0.5)*\stringgap}
    \path ($(std.north west)!\xoffset!(std.north east)$) node (topb\i) {};
    \path (topb\i) ++(0,3.5cm) node (topa\i) {};
    \path ($(std.south west)!\xoffset!(std.south east)$) node (bota\i) {};
    \path (bota\i) 
                   ++(0,-1.5cm) node (botc\i) {};
}
\pgfmathsetmacro{\dotoffset}{(\strings+0.75)*\stringgap}
\path ($(std.west)!\dotoffset!(std.east)$) node (dot) {};
\path (dot.center |- botc6.center) node (botdot) {};
\path ($(topa5.center)!0.25!(topb5.center)$) node (turn1) {};
\path (turn1.center -| botdot.center) node (turn2) {};
\path ($(turn1.center)!0.5!(turn2.center)$) node (turn) {};
\path ($(topa5.center)!0.75!(topb5.center)$) node (dotnew) {};
\path ($(topa5.center)!0.375!(topb5.center)$) node[dot,fill=red] (barbella) {};
\path ($(topa5.center)!0.625!(topb5.center)$) node[dot,fill=red] (barbellb) {};
\foreach \i in {1,...,6}
{
    \pgfmathsetmacro{\xoffset}{(\i-0.5)*\stringgap}
    \draw (bota\i.center) edge[string] (botc\i.center);
}
\draw (bota3.center) edge[string,red] (botc3.center);
\draw (dot.center) edge[string,red] (botdot.center);
\foreach \i in {2,...,4}
{
    \pgfmathsetmacro{\xoffset}{(\i-0.5)*\stringgap}
    \draw (topa\i.center) edge[string] (topb\i.center);
}
\draw[corner, string, red] (topa5.center) |- (turn.center) -| (botdot.center);
\draw[corner,string,red] (topb5.center) |- (dotnew.center -| botdot.center);
\draw[string,red] (barbella.center) edge (barbellb.center);
\end{tikzpicture}
\end{equation*}
The second term on the right-hand side vanishes modulo lower terms as it factors through $B_{\underline{ws}}$.
The last term is already a light leaves map with an extra polynomial $\alpha_s$. 
So it is enough to reduce the first term. 
By the Jones--Wenzl relation \cite[(5.7)]{EW-SoergelCalc} dots ``propagate'' down through braids (see also \cite[(4.4)]{matrixrecursion}). 
Thus we can push the dot through any rex move until it either reaches the bottom of the diagram or it reaches a trivalent vertex.
In the first case we obtain a new light leaves map with a $\mathrm{U}0$-strand in place of a $\mathrm{U}1$-strand.
In the second case, we obtain a ``birdcage''
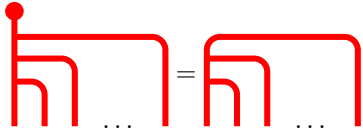
\begin{equation*}
\begin{tikzpicture}[scale=0.4,baseline=(origin.center)]
\coordinate (origin) at (0,0) {};
\path (origin) ++(0,-1.5cm) coordinate (bot1) {}
               ++(1cm,0) coordinate (bot2) {}
               ++(1cm,0) coordinate (bot3) {}
               ++(1.5cm,0) node (ellipses) {$\dotso$}
               ++(1.5cm,0) coordinate (bot4) {};
\path (origin) ++(0,3cm) coordinate (top1) {};
\path ($(bot1.center)!0.33!(top1.center)$) coordinate (along2);
\path ($(bot1.center)!0.5!(top1.center)$) coordinate (along3);
\path ($(bot1.center)!0.66!(top1.center)$) coordinate (along4);
\path ($(bot1.center)!0.85!(top1.center)$) node[dot,fill=red] (dot) {};
\draw[string,red] (bot1.center) edge (dot.center);
\draw[corner,string,red] (bot2.center) |- (along2.center)
                  (bot3.center) |- (along3.center)
                  (bot4.center) |- (along4.center);
\end{tikzpicture}
=
\begin{tikzpicture}[scale=0.4,baseline=(origin.center)]
\coordinate (origin) at (0,0) {};
\path (origin) ++(0,-1.5cm) coordinate (bot1) {}
               ++(1cm,0) coordinate (bot2) {}
               ++(1cm,0) coordinate (bot3) {}
               ++(1.5cm,0) node (ellipses) {$\dotso$}
               ++(1.5cm,0) coordinate (bot4) {};
\path (origin) ++(0,3cm) coordinate (top1) {};
\path ($(bot1.center)!0.33!(top1.center)$) coordinate (along2);
\path ($(bot1.center)!0.5!(top1.center)$) coordinate (along3);
\path ($(bot1.center)!0.66!(top1.center)$) ++(1cm,0) coordinate (turn);
\draw[corner,string,red] (bot2.center) |- (along2.center)
                  (bot3.center) |- (along3.center)
                  (bot4.center) |- (turn.center) -| (bot1.center);
\end{tikzpicture}
\end{equation*}
which also forms part of a light leaves map: the first strand of the birdcage is $\mathrm{U}1$, the last strand of the birdcage is $\mathrm{D}1$, and any intermediate strands are $\mathrm{D}0$.
\end{proof}

This shows that the light leaves maps for a basis for certain $\Hom$-spaces modulo lower terms.
Applying the same argument as \cite[\S 7.3]{EW-SoergelCalc} shows that the double leaves form bases for the entire $\Hom$-spaces of $\mathcal{H}_{\rm BS}$.
\end{proof}

\begin{rem}
The situation for diagrammatic \emph{singular} Hecke categories is completely different. 
In this setting the invariant rings $R^I$ for all \defnemph{finitary} $I \subseteq S$ (i.e.~for which $W_I$ is finite) play a key role. 
For all finitary $I \subset J$, the diagrammatic relations themselves posit the existence of a Demazure operator $\partial_J^I : R^I \rightarrow R^J$ which gives $R^J \subset R^I$ a Frobenius extension structure (see e.g.~\cite[\S 24.2.1]{emtw}. 
In particular, these Demazure operators must all be surjective for the diagrammatic singular Hecke category to be well defined. 

Unlike in the regular setting this condition naturally places restrictions on the base ring $\Bbbk$ and the Cartan matrix $A$. 
To see this, suppose $V$ is a realization for which the Demazure operators $\{\partial_J^I : I \subset J \text{ finitary}\}$ are all surjective. 
There is a natural morphism $\phi:V \rightarrow V_{A,{\rm sc}}$ by Lemma \ref{lem:adscuniversality}, which extends to a morphism $\phi:R \rightarrow R_{A,{\rm sc}}$ between the corresponding symmetric algebras.
One can check that $\partial_J^I = (\partial_J^I)_{A,{\rm sc}} \circ \phi$, where $(\partial_J^I)_{A,{\rm sc}}$ denotes the corresponding Demazure operator for $V_{A,{\rm sc}}$. 
This means that $V_{A,{\rm sc}}$ must have surjective Demazure operators too. 
The problem of determining precisely when this can happen dates back to Demazure's original paper introducing Demazure operators \cite{demazure}!
For the standard Cartan matrices of Weyl groups, he carefully showed that $2$ must be invertible outside of types $A$ and $C$, $3$ must be invertible in types $E$ and $F$, and $5$ must be invertible in type $E_8$.
\end{rem}

\section{Applications}

We conclude this paper with some applications of cinched Bott--Samelson and cinched Hecke categories. 
Many of these results have known analogues for $\mathcal{D}_{\rm BS}$ and $\mathcal{D}$ (with the added assumption of coroot surjectivity); the primary utility of the cinching construction is to simplify the proofs. 
As above $V$ is a $\Bbbk$-realization from which the categories $\mathcal{H}_{\rm BS}$, $\mathcal{H}$, $\mathcal{D}_{\rm BS}$, and $\mathcal{D}$ have been constructed.

The following is a version of \cite[Remark~3.19]{EW-SoergelCalc} for cinched Bott--Samelson categories.

\begin{lem} \label{lem:changerlz}
Let $V \rightarrow V'$ be a morphism of realizations, with symmetric algebras $R$ and $R'$ respectively.
The induced monoidal functor $\mathcal{H}_{\rm BS} \rightarrow \mathcal{H}'_{\rm BS}$ gives rise to an isomorphism $\mathcal{H}_{\rm BS} \cong R' \otimes_{R} \mathcal{H}_{\rm BS}$ of $\Bbbk$-linear categories, i.e.
\begin{equation*}
\Hom_{\mathcal{H}'_{\rm BS}}(B_{\underline{x}},B_{\underline{y}}) \cong R' \otimes_R \Hom_{\mathcal{H}_{\rm BS}}(B_{\underline{x}},B_{\underline{y}})
\end{equation*}
for any expressions $\underline{x},\underline{y}$ in $S$.
If $\Bbbk$ is a henselian local ring, then the Karoubi envelope of this functor maps indecomposable objects to indecomposable objects. 
\end{lem}

\begin{proof}
The induced monoidal functor maps light leaves to light leaves. 
Since the double leaves form bases for the $\Hom$-spaces this proves the first part.
It also shows that the local intersection forms (see e.g.~\cite[Definition~11.74]{emtw}) for the two categories coincide. 
The second part follows because local intersection forms determine the decompositions of the Bott--Samelson objects into direct sums of indecomposable objects \cite[Corollary~11.75]{emtw}
\end{proof}

We can use this to say a little more about the relationship between the cinched Bott--Samelson category $\mathcal{H}_{\rm BS}$ and the Elias--Williamson diagrammatic category $\mathcal{D}_{\rm BS}$.

\begin{prop}
If $V$ is coroot surjective, then $\mathcal{H}_{\rm BS}$ is equivalent to the Elias--Williamson diagrammatic category $\mathcal{D}_{\rm BS}$.
More generally, let $V \hookrightarrow V'$ be an embedding of $V$ into a coroot surjective realization $V'$, which extends to an embedding $R \hookrightarrow R'$ of their symmetric algebras. 
Let $\mathcal{D}'_{\rm BS}$ denote the Elias--Williamson diagrammatic category constructed from $V'$. 
Then there is an embedding $\mathcal{H}_{\rm BS} \rightarrow \mathcal{D}'_{\rm BS}$ which induces an isomorphism $\mathcal{D}'_{\rm BS} \cong R' \otimes_R \mathcal{H}_{\rm BS}$. 
\end{prop}

In other words, the cinched Bott--Samelson category $\mathcal{H}_{\rm BS}$ is equivalent to the usual Elias--Williamson diagrammatic category $\mathcal{D}_{\rm BS}$ whenever the latter is well behaved.
Moreover $\mathcal{H}_{\rm BS}$ is always a monoidal subcategory of a well-behaved Elias--Williamson diagrammatic category $\mathcal{D}'_{\rm BS}$. 

\begin{proof}
If $V$ is coroot surjective, then the cinching relation already holds in $\mathcal{D}_{\rm BS}$ \cite[(5.14)--(5.15)]{EW-SoergelCalc}. 

For general $V$, let $V'$ be a coroot surjective realization with an embedding $V \hookrightarrow V'$, and let $\mathcal{H}'_{\rm BS}$ denote the corresponding cinched Bott--Samelson category.
The composition of the induced functor $\mathcal{H}_{\rm BS} \rightarrow \mathcal{H}'_{\rm BS}$ with the equivalence $\mathcal{H}'_{\rm BS} \xrightarrow{\sim} \mathcal{D}'_{\rm BS}$ gives the $\Hom$-formula above.
\end{proof}

Next, we recall that a consequence of Soergel's categorification theorem is that the indecomposable objects give rise to a \defnemph{canonical basis} of the Hecke algebra. 
For example, the \defnemph{$p$-canonical} or \defnemph{$p$-Kazhdan--Lusztig basis} arises from the Hecke category for a Cartan realization of a finite or affine Weyl group over a field $\Bbbk$ of characteristic $p$. 
In principle such canonical bases depend on the realization, and indeed there are numerous examples which showcase dependency on the characteristic $p$ or the Cartan matrix $A$. 
In fact, these essentially turn out to be the only things the canonical basis can depend on. 
This has been observed before for the $p$-canonical basis \cite[Remark~10.17(2)]{rw-tilting}. 
The only potential difficulty in formulating a more general result is describing precisely what it means for two realizations over different base rings to have ``the same'' Cartan matrix.

\begin{prop}
Suppose $\Bbbk$ is a henselian local ring with residue field $\mathbb{F}$. 
Let $A$ denote the Cartan matrix of the realization $V$, and let $\mathbb{F}'$ be the subfield of $\mathbb{F}$ generated by its prime subfield and the images of the entries $a_{st}$ of $A$ in $\mathbb{F}$. 
The canonical basis of the Hecke algebra arising from $\mathcal{H}$ depends only on $\mathbb{F}'$ and image of $A$ in $\mathbb{F}'$. 
\end{prop}

\begin{proof}
If $B$ is an indecomposable object in $\mathcal{H}$, then its image in $\mathbb{F} \otimes_{\Bbbk} \mathcal{H}$ is also indecomposable by the usual idempotent lifting argument.
The latter is the cinched Hecke category for the $\mathbb{F}$-realization $\mathbb{F} \otimes_{\Bbbk} V$, so without loss of generality we may assume that $\Bbbk$ is a field. 

Let $\Bbbk'$ be the subfield of $\Bbbk$ generated by the prime subfield and the images of the entries of the Cartan matrix, and let $V'$ denote the restriction of $V$ to a $\Bbbk'$-module. 
Clearly $V'$ is a $\Bbbk'$-realization, and for the corresponding cinched Bott--Samelson category $\mathcal{H}_{\rm BS}'$ we have $\mathcal{H}_{\rm BS}=\Bbbk \otimes_{\Bbbk'} \mathcal{H}'_{\rm BS}$, and similarly for the local intersection forms of the two categories. 
Recall that the decomposition of each Bott--Samelson object into a direct sum of indecomposable objects is determined by the ranks of these forms. 
As the rank of a matrix does not change after field extension, this shows that indecomposable objects in $\mathcal{H}'$ remain indecomposable in $\mathcal{H}$.
So without loss of generality we may further assume that $\Bbbk=\Bbbk'$.

To complete the proof we consider the morphism $V_{A,{\rm ad}} \rightarrow V$, which gives rise to a monoidal functor $\mathcal{H}_{A,{\rm ad}} \rightarrow \mathcal{H}$ by Lemma \ref{lem:changerlz}.
This functor maps indecomposable objects to indecomposable objects, so the canonical bases arising from both cinched Hecke categories must coincide.
\end{proof}

Finally we apply our cinched constructions to the parabolic anti-spherical setting, to show that any realization can be used to categorify the anti-spherical module of the Hecke algebra. 

Fix some subset $I \subseteq S$. 
We write $W_I$ to denote the subgroup generated by $I$, and $\prescript{I}{}{W}$ for the set of minimal length representatives for the right cosets $W_I \backslash W$. 
An expression $\underline{x}$ in $S$ is called an \defnemph{$I$-sequence} if it starts with some element $s \in I$. 
We define the \defnemph{(cinched) anti-spherical Bott--Samelson category $\mathcal{N}_{\rm BS}$} to be the quotient of $\mathcal{H}_{\rm BS}$ by the ideal generated by $\{B_{\underline{x}} : \underline{x} \text{ is an $I$-sequence}\}$.
It is a right module category for $\mathcal{H}_{\rm BS}$, i.e.~there is a right monoidal action of $\mathcal{H}_{\rm BS}$ on $\mathcal{N}_{\rm BS}$. 
We call the Karoubi envelope of $\mathcal{N}_{\rm BS}$ the \defnemph{(cinched) anti-spherical category $\mathcal{N}$}. 

Recall that the light leaves construction for $\mathcal{H}_{\rm BS}$ generalizes to $\mathcal{N}_{\rm BS}$, with anti-spherical light leaves and double leaves \cite[Definition~5.1]{LW-antispher}.

\begin{thm}[{cf.~\cite[Theorem~5.3]{LW-antispher}}] \label{thm:antispherLLbasis}
The anti-spherical light leaves and double leaves constructions yield bases for the $\Hom$-spaces of the cinched anti-spherical Bott--Samelson category $\mathcal{N}_{\rm BS}$.
\end{thm}

Libedinsky--Williamson proved this when the underlying realization $V$ is Demazure surjective and satisfies the ``parabolic property'' \cite[\S 2.3]{LW-antispher}.

\begin{proof}
Let $U$ denote the root subrealization of $V$, and let $K$ be a complement for $U$ inside $V$. 
Choose a basis $f_1,\dotsc,f_r$ of $K$. 
Let $\mathbb{O}$ be the subring of $\overline{\QQ}$ generated by the algebraic integers $4\cos^2(\pi/m_{st})$ for all $s,t \in S$ such that $m_{st}<\infty$.
By definition, for every $s,t \in S$ the product $\langle \alpha_s^\vee,\alpha_t\rangle \langle \alpha_t^\vee, \alpha_s \rangle$ in $\Bbbk$ can be lifted to $\mathbb{O}$.
Next we use $\mathbb{O}$ to define a ring
\begin{equation*}
\Bbbk'=\frac{\mathbb{O}[x_{s,t},y_{s,i}: s,t \in S,\ 1 \leq i \leq r]}{(x_{s,t}x_{t,s} - 4\cos^2(\pi/m_{st}) \text{ for all $s,t \in S$ with $m_{st}<\infty$})} \text{.}
\end{equation*}
Now we define a new $\Bbbk'$-realization
\begin{equation*}
V' = \left(\bigoplus_{s \in S} \Bbbk'\alpha'_s \right) \oplus \left(\bigoplus_{i=1}^r \Bbbk' f'_i\right)
\end{equation*}
with linearly independent roots $\{\alpha'_s : s \in S\}$ by setting 
\begin{align*}
\langle (\alpha'_s)^\vee, f'_i \rangle & = y_{s,i} \text{ for all $s \in S$ and $1 \leq i \leq r$,} \\
\langle (\alpha'_s)^\vee, \alpha'_t \rangle & =x_{s,t} \text{ for all $s,t \in S$.}
\end{align*}

Consider the specialization of $V'$ where $x_{s,t}=-2\cos(\pi/m_{st})$ for all $s,t \in S$ (suitably interpreting $\cos(\pi/\infty)=\cos 0=1$). 
The root subrealization of this specialization is equivalent to the geometric representation, as defined in \cite[\S 5.3]{humphreys}. 
As this representation is faithful, the original realization $V'$ must also be faithful, and therefore satisfies the parabolic property. 
Thus by \cite[Theorem~5.3]{LW-antispher} the anti-spherical light leaves form bases for the relevant $\Hom$-spaces in $\mathcal{N}'_{\rm BS}$, where $\mathcal{N}'_{\rm BS}$ is the anti-spherical Bott--Samelson category defined by $V'$. 
But $V'$ also specializes to $V$ in the obvious way, by setting $x_{s,t}=\langle \alpha_s^\vee,\alpha_t\rangle$ and $y_{s,i}=\langle \alpha_s^\vee,f_i\rangle$ for all $s,t \in S$ and $1 \leq i \leq r$. 
Since the bases for the $\Hom$-spaces remain bases after specialization, this proves the result.
\end{proof}

\begin{cor}[{cf.~\cite[Theorem~6.2]{LW-antispher}}] \label{cor:antispherSoergelcatthm}
Let $\Bbbk$ be a henselian local ring. 
The category $\mathcal{N}$ categorifies the anti-spherical module of the Hecke algebra. 
\end{cor}

\begin{rem}
A simpler version of Theorem \ref{thm:antispherLLbasis} first appeared in \cite[Theorem~1.10]{BHN-modularWeylKac}. 
It was used to show that the anti-spherical category for Cartan realizations and geometric realizations is well behaved in positive characteristic \cite[Example~1.11]{BHN-modularWeylKac}. 
A related idea to justify defining the $p$-canonical basis via $p$-modular systems was described in \cite[\S 3.3]{gjw-calcpcanbasis}.
\end{rem}

\begin{rem} \label{rem:locpres}
The linear independence argument in Claim \ref{clm:LLlinindep} is completely new, and applies to the Elias--Williamson diagrammatic category $\mathcal{D}_{\rm BS}$ with no change.
As written it is only relevant when $2$ vanishes in $\Bbbk$, but the same type of argument works in other settings where roots vanish, such as anti-spherical Hecke categories (as in Theorem \ref{thm:antispherLLbasis}) and cyclotomic Hecke categories (analogous to categories of Soergel modules). 

Indeed, it is worth understanding how to reinterpret localization in such settings.
Carrying over the notation from Claim \ref{clm:LLlinindep} let $\Lambda' : \mathcal{H}'_{\rm BS} \rightarrow (\Omega_{Q'} W)_{\oplus}$ be the localization functor for $\mathcal{H}'_{\rm BS}$, as defined in \cite{ew-loccalc}. 
Let $\mathcal{M}'$ denote the image of $\Lambda'$ inside $(\Omega_{Q'} W)_{\oplus}$, which is isomorphic to $\mathcal{H}'_{\rm BS}$ because localization is faithful.
It has a concrete presentation in terms of matrices with entries in the fraction field $Q'$ of $R'$. 
Now let 
\begin{equation}
\Lambda = R \otimes_{R'} \Lambda' : R \otimes_{R'} \mathcal{H}'_{\rm BS} \xrightarrow{\sim} R \otimes_{R'} \mathcal{M}'
\end{equation}
This gives a faithful presentation of $\mathcal{H}_{\rm BS}$ in terms of a matrix algebra in which the scalars have been changed. 
Note in particular that the target of $\Lambda$ does \emph{not} lie in $Q \otimes_{Q'}(\Omega_{Q'} W)_{\oplus} = (\Omega_{Q} W)_{\oplus}$, where $Q$ denotes the fraction field of $R$. 
In other words, we cannot interpret the image of $\Lambda$ as a matrix algebra with entries in $Q$. 
For example, if $\alpha_s=0$ in $R$ then
\begin{equation*}
1 \otimes 
\begin{pmatrix}
\alpha_s & 0
\end{pmatrix} = \alpha_s \otimes 
\begin{pmatrix}
1 & 0
\end{pmatrix}
=0
\end{equation*}
in $Q \otimes_{Q'} (\Omega_{Q'} W)_{\oplus}$, yet we still have
\begin{equation*}
\Lambda({\rm dot}_s)= 1 \otimes 
\begin{pmatrix}
\alpha_s & 0
\end{pmatrix} \neq 0
\end{equation*}
since $\begin{pmatrix} 1 & 0 \end{pmatrix} \notin \mathcal{M}'$. 

The moral here is that the \emph{process} of localizing (i.e.~inverting roots) in Hecke categories is far less important than the resulting localization \emph{representation}. 
Once one obtains the representation $\Lambda$ by any means, the expected results (light leaves basis, Soergel categorification, etc.) follow as usual. 
The germs of this idea are already apparent in the work of Abe \cite{abe-bimodhecke} and Elias--Williamson \cite{ew-loccalc}, which both emphasized the role of the localization representation (in seemingly different ways).
\end{rem}

\printbibliography
\end{document}